\documentclass[11pt]{amsart}

\usepackage{amssymb,enumerate}
\usepackage{amsmath}
\usepackage{bbm}           % used for blackboard 1
\usepackage{bm}
\usepackage{eso-pic,graphicx}
\usepackage{tikz}
\usepackage[colorlinks=true, pdfstartview=FitV, linkcolor=blue, citecolor=blue, urlcolor=blue]{hyperref}

\usepackage{graphicx}
\usepackage{pslatex}
\usepackage{amsmath,amsfonts}
\usepackage{rotating}
\usepackage{amssymb}
\usepackage{verbatim}
\usepackage{rotating}
\usepackage{mathrsfs}

\makeatletter
\def\Ddots{\mathinner{\mkern1mu\raise\p@
\vbox{\kern7\p@\hbox{.}}\mkern2mu
\raise4\p@\hbox{.}\mkern2mu\raise7\p@\hbox{.}\mkern1mu}}
\makeatother

\def\XXint#1#2#3{{\setbox0=\hbox{$#1{#2#3}{\int}$}
\vcenter{\hbox{$#2#3$}}\kern-.5\wd0}}

\begin{document}
\newtheorem{theorem}{Theorem}
\newtheorem{proposition}[theorem]{Proposition}
\newtheorem{conjecture}[theorem]{Conjecture}
\def\theconjecture{\unskip}
\newtheorem{corollary}[theorem]{Corollary}
\newtheorem{lemma}[theorem]{Lemma}
\newtheorem{claim}[theorem]{Claim}
\newtheorem{sublemma}[theorem]{Sublemma}
\newtheorem{observation}[theorem]{Observation}
\theoremstyle{definition}
\newtheorem{definition}{Definition}
\newtheorem{notation}[definition]{Notation}
\newtheorem{remark}[definition]{Remark}
\newtheorem{question}[definition]{Question}
\newtheorem{questions}[definition]{Questions}
\newtheorem{example}[definition]{Example}
\newtheorem{problem}[definition]{Problem}
\newtheorem{exercise}[definition]{Exercise}
 \newtheorem{thm}{Theorem}
 \newtheorem{cor}[thm]{Corollary}
 \newtheorem{lem}{Lemma}[section]
 \newtheorem{prop}[thm]{Proposition}
 \theoremstyle{definition}
 \newtheorem{dfn}[thm]{Definition}
 \theoremstyle{remark}
 \newtheorem{rem}{Remark}
 \newtheorem{ex}{Example}
 \numberwithin{equation}{section}
%%%%%%%%%%%%% Abbreviation of symbols %%%%%%%%%%%%%%%%
\def\C{\mathbb{C}}
\def\R{\mathbb{R}}
\def\Rn{{\mathbb{R}^n}}
\def\Rns{{\mathbb{R}^{n+1}}}
\def\Sn{{{S}^{n-1}}}
\def\M{\mathbb{M}}
\def\N{\mathbb{N}}
\def\Q{{\mathbb{Q}}}
\def\Z{\mathbb{Z}}
\def\F{\mathcal{F}}
\def\L{\mathcal{L}}
\def\S{\mathcal{S}}
\def\supp{\operatorname{supp}}
\def\essi{\operatornamewithlimits{ess\,inf}}
\def\esss{\operatornamewithlimits{ess\,sup}}
%%%%%%%%%%%%%%%%%%%%%%%%%%%%%%%%%%%%%%%%%%%%%%%%%%%%%

\numberwithin{equation}{section}
\numberwithin{thm}{section}
\numberwithin{theorem}{section}
\numberwithin{definition}{section}
\numberwithin{equation}{section}

\def\earrow{{\mathbf e}}
\def\rarrow{{\mathbf r}}
\def\uarrow{{\mathbf u}}
\def\varrow{{\mathbf V}}
\def\tpar{T_{\rm par}}
\def\apar{A_{\rm par}}

\def\reals{{\mathbb R}}
\def\torus{{\mathbb T}}
\def\scriptm{{\mathcal T}}
\def\heis{{\mathbb H}}
\def\integers{{\mathbb Z}}
\def\z{{\mathbb Z}}
\def\naturals{{\mathbb N}}
\def\complex{{\mathbb C}\/}
\def\distance{\operatorname{distance}\,}
\def\support{\operatorname{support}\,}
\def\dist{\operatorname{dist}\,}
\def\Span{\operatorname{span}\,}
\def\degree{\operatorname{degree}\,}
\def\kernel{\operatorname{kernel}\,}
\def\dim{\operatorname{dim}\,}
\def\codim{\operatorname{codim}}
\def\trace{\operatorname{trace\,}}
\def\Span{\operatorname{span}\,}
\def\dimension{\operatorname{dimension}\,}
\def\codimension{\operatorname{codimension}\,}
\def\nullspace{\scriptk}
\def\kernel{\operatorname{Ker}}
\def\ZZ{ {\mathbb Z} }
\def\p{\partial}
\def\rp{{ ^{-1} }}
\def\Re{\operatorname{Re\,} }
\def\Im{\operatorname{Im\,} }
\def\ov{\overline}
\def\eps{\varepsilon}
\def\lt{L^2}
\def\diver{\operatorname{div}}
\def\curl{\operatorname{curl}}
\def\etta{\eta}
\newcommand{\norm}[1]{ \|  #1 \|}
\def\expect{\mathbb E}
\def\bull{$\bullet$\ }

\def\blue{\color{blue}}
\def\red{\color{red}}

\def\xone{x_1}
\def\xtwo{x_2}
\def\xq{x_2+x_1^2}
\newcommand{\abr}[1]{ \langle  #1 \rangle}

\newcommand{\Norm}[1]{ \left\|  #1 \right\| }
\newcommand{\set}[1]{ \left\{ #1 \right\} }
\newcommand{\ifou}{\raisebox{-1ex}{$\check{}$}}
\def\one{\mathbf 1}
\def\whole{\mathbf V}
\newcommand{\modulo}[2]{[#1]_{#2}}
\def \essinf{\mathop{\rm essinf}}
\def\scriptf{{\mathcal F}}
\def\scriptg{{\mathcal G}}
\def\scriptm{{\mathcal M}}
\def\scriptb{{\mathcal B}}
\def\scriptc{{\mathcal C}}
\def\scriptt{{\mathcal T}}
\def\scripti{{\mathcal I}}
\def\scripte{{\mathcal E}}
\def\scriptv{{\mathcal V}}
\def\scriptw{{\mathcal W}}
\def\scriptu{{\mathcal U}}
\def\scriptS{{\mathcal S}}
\def\scripta{{\mathcal A}}
\def\scriptr{{\mathcal R}}
\def\scripto{{\mathcal O}}
\def\scripth{{\mathcal H}}
\def\scriptd{{\mathcal D}}
\def\scriptl{{\mathcal L}}
\def\scriptn{{\mathcal N}}
\def\scriptp{{\mathcal P}}
\def\scriptk{{\mathcal K}}
\def\frakv{{\mathfrak V}}
\def\C{\mathbb{C}}
\def\D{\mathcal{D}}
\def\R{\mathbb{R}}
\def\Rn{{\mathbb{R}^n}}
\def\rn{{\mathbb{R}^n}}
\def\Rm{{\mathbb{R}^{2n}}}
\def\r2n{{\mathbb{R}^{2n}}}
\def\Sn{{{S}^{n-1}}}
\def\M{\mathbb{M}}
\def\N{\mathbb{N}}
\def\Q{{\mathcal{Q}}}
\def\Z{\mathbb{Z}}
\def\F{\mathcal{F}}
\def\L{\mathcal{L}}
\def\G{\mathscr{G}}
\def\ch{\operatorname{ch}}
\def\supp{\operatorname{supp}}
\def\dist{\operatorname{dist}}
\def\essi{\operatornamewithlimits{ess\,inf}}
\def\esss{\operatornamewithlimits{ess\,sup}}
\begin{comment}
\def\scriptx{{\mathcal X}}
\def\scriptj{{\mathcal J}}
\def\scriptr{{\mathcal R}}
\def\scriptS{{\mathcal S}}
\def\scripta{{\mathcal A}}
\def\scriptk{{\mathcal K}}
\def\scriptp{{\mathcal P}}
\def\frakg{{\mathfrak g}}
\def\frakG{{\mathfrak G}}
\def\boldn{\mathbf N}
\end{comment}
\author{Qingying Xue}
\address{
         School of Mathematical Sciences \\
         Beijing Normal University \\
         Laboratory of Mathematics and Complex Systems \\
         Ministry of Education \\
         Beijing 100875 \\
         People's Republic of China}
\email{qyxue@bnu.edu.cn}

\author{K\^{o}z\^{o} Yabuta}
\address{K\^{o}z\^{o} Yabuta
\\ Research Center for Mathematical Sciences
\\
Kwansei Gakuin University
\\
Gakuen 2-1, Sanda 669-1337
\\
Japan } \email{kyabuta3@kwansei.ac.jp}
\thanks{
The first author was partly supported by NSFC (Nos. 11471041, 11671039) and NSFC-DFG (No. 11761131002). The third
named author was supported partly by Grant-in-Aid for Scientific Research (C) Nr.
15K04942, Japan Society for the Promotion of Science.\\
 \indent Corresponding author: Qingying Xue \indent Email: qyxue@bnu.edu.cn}
\author{Jingquan Yan}
\address{Jingquan Yan\\
         School of Mathematical and Computational Science\\ Anqing Normal University\\ Anqing, 246133\\ P. R. China}
\email{yjq20053800@yeah.net}


\date{\today}
\title[Weighted Fr\'{e}chet-Kolmogorov theorem and compactness ]{Weighted Fr\'{e}chet-Kolmogorov theorem and compactness of vector-valued multilinear operators}
\maketitle

\begin{abstract}In this paper, we establish a weighted version of the well-known Fr\'{e}chet-Kolmogorov theorem, which holds for weights beyond $A_\infty$. This weighted theory extends the previous known unsatisfactory results in the terms of relaxing the index to the natural range. As applications, we obtain the weighted compactness theory for the commutators of multilinear vector-valued Calder\'{o}n-Zygmund type operators, including the commutators of multilinear Littlewood-Paley type operators. It is worthy to pointing out that the commutators we considered contain almost all the commutators formerly studied in this literature.
\end{abstract}

\section{Introduction}
 The well-known {Fr\'{e}chet-Kolmogorov theorem (\cite{R}, \cite[p. 275]{Yo}) was first proved by Riesz \cite{R} in 1933. It states that:\medskip
 	 	
 	\quad\hspace{-20pt}{\bf Theorem A} (\cite{R}, \cite[p. 275]{Yo}). Let $1\le p<\infty$ and $F$ be a subset in $L^p(\R^n)$. $F$ is sequentially compact in $L^p(\R^n)$ if and only if the following three conditions are satisfied:
		\begin{enumerate}[{\rm(i)}]
			\item %(i)
			$\sup\limits_{f\in F}\|f\|_{L^p}< \infty$;
			\item %(ii)
			$\lim\limits_{N\rightarrow\infty}\sup\limits_{f\in F}\int_{|x|\geq N}|f(x)|^pdx=0;$
			\item %(iii)
			$\lim\limits_{|t|\rightarrow0}\sup\limits_{f\in F}\int_{\mathbb{R}^n}|f(\cdot+t)-f(\cdot)|^pdx=0.$
		\end{enumerate}
The weighted Fr\'{e}chet-Kolmogorov theorem was given by Clop and Cruz in \cite{ClC}.
\medskip

\quad\hspace{-20pt}{\bf Theorem B} (\cite{ClC}).\label{thm-comp-lpomega-clop}
	Let $1\leq p<\infty$, $w\in A_p$, $F\subset L^p(w)$. Then $F$ is a compact set
	in $L^p(w)$ if and only if the following three conditions hold:
	\begin{enumerate}
		\item[\emph{(i)}]$\sup_{f\in F}\|f\|_{L^p(w)}<\infty;$
		\item[\emph{(ii)}]$\lim_{N\rightarrow\infty}\sup_{f\in F}\int_{|x|\geq N}|f(x)|^pw(x)dx=0;$
		\item[\emph{(iii)}]$\lim_{|t|\rightarrow0}\sup_{f\in F}\int_{\mathbb{R}^n}|f(\cdot+t)-f(\cdot)|^pw(x)dx=0.$
	\end{enumerate}

It was actually only pointed out in \cite{ClC} that conditions (i)-(iii) in Theorem B were sufficient for the set $A$ to be a compact set in $L^p(\omega)$ when $p>1$. But it is easy to see that their argument also holds for $p=1$. 
However, the necessity part (iii) in Theorem B does not follow, in general. 
Later we shall give examples showing unnecessity of the condition (iii). 

\medskip

It is natural to ask whether  Fr\'{e}chet-Kolmogorov theorem is true or not for $0<p<1$.
In 1951, Tsuji \cite{Ts} showed that the unweighted
Fr\'{e}chet-Kolmogorov theorem can be extended to $0<p<1$. This article hasn't received much attention for a long time.
Based on this somehow unnoticed paper \cite{Ts}, the first aim of this paper is  to show that the weighted Fr\'{e}chet-Kolmogorov theorem can also be extended to the case $0<p<1$. Moreover, we actually showed that the following weighted Fr\'{e}chet-Kolmogorov theorem holds even for more general weights than $A_\infty$, which could be of interests in its own.
\begin{thm}\label{thm-comp-lpomega}
	Let $w$ be a weight on $\R^n$. Assume that $w^{-1/(p_0-1)}$ is also a weight
	on $\R^n$ for some $p_0>1$. Let $0<p<\infty$ and $F$ be a subset in $L^p(w)$.
	Then $F$ is sequentially compact in $L^p(w)$ if the following three
	conditions are satisfied:
	\begin{enumerate}[{\rm(i)}]
		\item %(i)
		$\sup\limits_{f\in F}\|f\|_{L^p(w)}< \infty$;
		\item %(ii)
		$\lim\limits_{N\rightarrow\infty}\sup\limits_{f\in F}\int_{|x|\geq N}|f(x)|^pw(x)dx=0;$
		\item %(iii)
		$\lim\limits_{|t|\rightarrow0}\sup\limits_{f\in F}\int_{\mathbb{R}^n}|f(\cdot+t)-f(\cdot)|^pw(x)dx=0.$
	\end{enumerate}
\end{thm}
With Theorem \ref{thm-comp-lpomega} in hand, the second main aim of this paper is to apply Theorem \ref{thm-comp-lpomega} in the study of compacetness of the commutators of Calder\'on-Zygmund type operators. We begin by recalling some known results.
\medskip

It was well known that the boundedness of the linear commutators of Calder\'on-Zygmund operators was given by Coifman, Rochberg and Weiss \cite{CRW} when the symbol is in $BMO(\R^n)$.
Later on, commutators of some classical operators, such as the Calder\'{o}n-Zygmund operators \cite{Uc}, the  Fourier multipliers \cite{CO}  and certain Littlewood-Paley square functions \cite{ChD}, have been shown that they are not only bounded on $L^p$ spaces but also compact when they are multiplied with functions in $CMO(\R^n)$.
\medskip

The bilinear commutators of the Calder\'on-Zygmund operator $T$ were first studied by P\'erez and Torres \cite{PT}, which are defined by
\begin{align}
[T, b]_1(f, g)(x)&= (T(bf,g)-bT(f,g))(x)\\
[T, b]_2 (f, g)(x)&=(T(f,bg)-bT(f,g))(x).
\end{align}
Related topics were later further studied by Tang \cite{tang}, Lerner et al \cite{LOPTT} and Xue \cite{Xue}.
The concept of compactness for a bilinear operator was first introduced by B\'{e}nyi and Torres \cite{BeT}. They proved that the commutators of bilinear Calder\'{o}n-Zygmund operator were also compact from $L^{p_1}\times L^{p_2}$ to $L^p$ if $\frac1{p_1}+\frac1{p_2}=\frac1p$, $1<p_1,p_2<\infty$ and $p\geq 1$.
Later on, this compactness property has been extended to the following operators: the maximal bilinear Calder\'{o}n-Zygmund operators \cite{DiMX}, the bilinear Fourier multiplier operators \cite{Hu} and the bilinear pseudodifferential operators \cite{BeO}. The working spaces have also been extended to weighted Lebesgue spaces \cite{BeDMT} and Morry spaces \cite{DiM}. 
\medskip

All the previously mentioned results of compactness need to assume that $p\geq 1$.
However, the boundedness of the commutators of the above operators shows that $p\geq1$ is unnecessary for the boundedness to be hold. This restriction comes from the application of Fr\'{e}chet-Kolmogorov theorem \cite{Yo}, as well as its weighted analogue \cite{ClC}.
However, Tsuji \cite{Ts} had already showed that the unweighted Fr\'{e}chet-Kolmogorov theorem can be extended to $p>0$.
Recently,  the region of $p$ was extended to the case $1/2<p\leq 1$ by  Torres, Xue, and Yan \cite{TXY}. It was shown that if $b\in CMO$, $1<p_1,p_2<\infty$ and $1/2< p < \infty$ with $1/p_1+1/p_2=1/p$, then the commutator defined in (1.1) or (1.2), $[T, b]_j: L^{p_1}\times L^{p_2}\rightarrow L^p$ is a compact bilinear operator.  Still more recently, Chaffee et al \cite{CCHTW} further proved that the commutators of certain kinds of homogeneous bilinear Calder\'on-Zygmund operators
enjoy the compactness property if and only if $b\in CMO$ and $p>1$.
\medskip

We will consider the compactness of the following generalized commutator of multilinear Calder\'{o}n-Zygmund operator.
\begin{dfn}\label{df1.2} (\cite{XuY})
Let $T$ be an $m$-linear Calder\'{o}n-Zygmund operator and $S$ be a finite subset of $Z^+\times \{1,\dots, m\}$.  For any
$f_j\in \mathscr{S}$, $j=1,\dots, m$,  the generalized commutator $T_{\vec b,S}$ of $T$ is defined by
\begin{equation*}\aligned
T_{\vec{b},S} (\vec{f})(x)=\int_{(\mathbb{R}^n)^m}\prod_{(i,j)\in
S} (b_i(x)-b_i(y_j))K(x,y_{1},\dots,y_{m})
\prod_{j=1}^mf_j(y_j)dy_1,\dots,dy_m.
\endaligned
\end{equation*}
\end{dfn}
$T_{\vec{b},S}$ is called to be the generalized commutator of $T$ due to its capacity and flexibility with respect to $S$. For example by choosing $S=\{(i,i): i\in\{1,\dots, m\}\}$, or $S=\{(j, j): j\in\{1,\dots, m\}\}$, then $T_{\vec{b},S}$ coincides with $T_{\vec{b}} (\vec{f})$ \cite{LOPTT} and $T_{\prod b}$ \cite{PPTT}, respectively. Other choice of selection may lead to new type of commutators of $T.$ 
\medskip

Instead of considering the compactness for the commutators of operator $T$, we will try to establish directly a compactness theory for the commutators of vector-valued multilinear operators. This is mainly because multilinear Littlewood-Paley operators, such as multilinear $g$-function, 
Marcinkiewicz integral and $g^*_{\lambda}$-function (see Section 3 for the definitions) can be regarded as multilinear vector-valued Calder\'on-Zygmund operators \cite{XuY2}, therefore, the extension of the compactness result to vector-valued operators does make sense  and it will enable us to study the generalized commutators of these multilinear square operators similarly defined as in Definition \ref{df1.2} and will enable us to show that the commutators of them are all compact. For more recent works related to the compactness of other operators,such like the bilinear
Fourier multipliers and the bilinear pseudodifferential operators, we refere the readers to \cite{TX19} and \cite{TXYY}.
\medskip

We begin with some definitions.
Let $0<r<\infty$. For any quasi-Banach space, denote $B_{r,X}=\{x\in X:\|x\|_X\leq r\}$.
Then there is a natural extension of the corresponding definition in \cite{BeT} ($m=2$).
\medskip

\begin{dfn}\label{dy1-ch7}
Let $X_j (j=1,\dots, m),Y$ be quasi-Banach spaces. A multilinear operator $T:X_1\times\dots\times X_m \rightarrow Y$ is called a compact operator, if $T(B_{1,X_1}\times\dots\times B_{1,X_m})$ is relatively compact in $Y$.
\end{dfn}

\begin{rem}
In the above definition, it is equivalent to require
$T(B_{1,X_1}\times\dots\times B_{1,X_m})$ to be sequentially compact in $Y$, as $Y$ is quasi-Banach space. 
\end{rem}

Let $B(X_1\times\dots\times X_m,Y)$ be the set of all bounded multilinear operators from $X_1\times\dots\times X_m$ to $Y$ and let
$K(X_1\times\dots\times X_m,Y)$ be the set of all compact operators from $X_1\times\dots\times X_m$ to $Y$. Then, we have 
%\begin{thm}\label{thm-closedness-com}
$K(X_1\times\dots\times X_m,Y)$ is closed in $B(X_1\times\dots\times X_m,Y)$.
%\end{thm}
To see this, when $Y$ is a Banach space and $m=2$, this property is shown in \cite{BeT}. As a matter of fact, it is sufficient to assume that $Y$ is quasi-Banach. The extension from $m=2$ to general $m$ needs only simple modification in the  proof given in \cite{BeT}.
%\end{rem}

\begin{dfn}
Let $1< p_j<\infty$,$j=1,\dots, m$, $\frac{1}{p}=\sum_{j=1}^m\frac{1}{p_j}$, $\vec\omega=(\omega_1,\dots,\omega_m)$, $A$ is a nonempty subset of $\{1,\dots, m\}$.
Denote $\frac1{p_A}=\sum_{j\in A}\frac{1}{p_j}$, $\nu_{\vec\omega,A}=\prod_{j\in A}\omega^{{p_A}/{p_j}}_j$.
We say that $\vec{\omega}$ satisfies the $A_{\vec{p},A}$ condition, or $\vec\omega_A\in A_{\vec p,A}$, if
\begin{equation*}\sup_Q{\bigg(\frac{1}{|Q|}\int_Q\nu_{\vec{\omega},A}\bigg)}^{\frac{1}{p_A}}
\prod_{j\in A}{\bigg(\frac{1}{|Q|}\int_Q\omega_j^{1-p_j'}\bigg)}^{\frac{1}{p_j'}} <\infty.
\end{equation*}
\end{dfn}

\begin{rem}
If $\omega_i\in A_{p_i}$, $i\in\{1,\dots, m\}$, then for any nonempty subset $A$, it follows that $\nu_{\vec\omega,A}\in A_{p_A|A|}$ and $\vec\omega_A\in A_{\vec p,A}$.
\end{rem}
We denote by $BMO(\R^n)$ the John-Nirenberg space of function of bounded mean oscillation endowed with its usual norm. The space of $C^\infty$ functions with compact support is denoted by $C^\infty_c$, we define
$$
CMO = \overline{{C^\infty_c}}^{BMO},
$$
the closure of $C^\infty_c$ in the $BMO$ norm.

Our second main result is as follows:.

%%%%%%%%%% Theorem 1.5 %%%%%%%%%%%%%%%%%%%%%%%%%%%%%%%%%%%
\begin{thm}\label{thmmain} Let $1< p_j<\infty$, 
$\frac{1}{p}=\sum_{j=1}^m\frac{1}{p_j}$, $S$ be a finite subset of 
$\mathbb Z^+\times \{1,\dots, m\}$, $S_j=\{i:(i,j)\in S\}$ for $1\le j\le m$. 
and $A=\{j\in \{1,2,\dots,m\}: (i,j)\in S\ \text{for some }i\in \mathbb Z^+\}$.
 Let $B$ be a Banach space, $T$ be a $B$-valued
$m$-linear Calder\'{o}n-Zygmund operator, and for any $i\in \bigcup_{j=1}^mS_j$, $b_i\in CMO$. If
$\vec{\omega}\in A_{\vec{p}}$, $\vec\omega_{A^c}\in A_{\vec p,A^c}$, and $\nu_{\vec\omega,A}\in A_{p_A|A|}$, then $T_{\vec b,S}$ is a compact operator from $L^{p_1} (\omega_1)\times\dots\times L^{p_m} (\omega_m)$ to $L^p_B(\nu_{\vec\omega})$.
\end{thm}
%%%%%%%%%% end Theorem 1.5 %%%%%%%%%%%%%%%%%%%%%%%%%%%%%%%%%%%

\begin{rem}
When $A=\{1,\dots, m\}$, $\nu_{\vec\omega,A}=\nu_{\vec\omega}$.
Then, in Theorem \ref{thmmain}, it is enough to assume that $\vec\omega\in A_{\vec p}$.
\end{rem}
Theorem \ref{thmmain} implies not only the compactness of generalized commutators of scalar-valued multilinear Calder\'{o}n-Zygmund operator, but also the compactness of generalized commutators of multilinear square operators, we surmmarize these results in two corollaries.
%%%%%%%%%%% Cor 1.6 %%%%%%%%%%%%%%%%%%%%%%%%%%%%%%%%%%%%
\begin{cor}\label{tm2}
Let $S$ and $S_j$ for $1\le j\le m$ be as in Theorem \ref{thmmain}. 
Let $1< p_j<\infty$, $j=1,\dots, m$, $\frac{1}{p}=\sum_{j=1}^m\frac{1}{p_j}$, 
$\vec\omega\in A_{\vec p}$, and for any $i\in \bigcup_{j=1}^mS_j$, 
$b_i\in CMO$. Then $T_{\vec b,S}$ is compact from 
$L^{p_1} (\omega_1)\times\dots\times L^{p_m} (\omega_m)$ to 
$L^p(\nu_{\vec\omega})$.
\end{cor}
%%%%%%%%%%% end Cor 1.6 %%%%%%%%%%%%%%%%%%%%%%%%%%%%%%%%%%%%
%%%%%%%%%%% Cor 1.7 %%%%%%%%%%%%%%%%%%%%%%%%%%%%%%%%%%%%
\begin{cor}
Let $S$ and $S_j$ for $1\le j\le m$ be as in Theorem \ref{thmmain}. 
Let $\mathcal{T}$ be any one of the following three operators: multilinear Littlewood-Paley $g$-function, multilinear
Marcinkiewicz integral, and multilinear $g^*_{\lambda}$-function (see Section 3 for the definitions).  Assume $1< p_{i}<
\infty$, and
$\frac{1}{p}=\frac{1}{p_{1}}+\cdots+\frac{1}{p_{m}}$, and for any $i\in \bigcup_{j=1}^mS_j$, $b_i\in CMO$. If
$\vec{\omega}\in A_{\vec{p}}$, then $\mathcal{T}_{\vec b,S}$ is a compact operator from $L^{p_1} (\omega_1)\times\dots\times L^{p_m} (\omega_m)$ to $L^p(\nu_{\vec\omega})$
\end{cor}
%%%%%%%%%%%% end Cor 1.7 %%%%%%%%%%%%%%%%%%%%%%%%%%%%%%%%%%%%
This paper will be organized as follows. We will recall some definitions and known results about multilinear Calder\'{o}n-Zygmund operators, multiple weights, generalized commutators and multilinear square operators in Section 2. Section 3 is devoted to present the theory of the vector-valued multilinear Calder\'{o}n-Zygmund operators, including their generalized commutators. The proofs of Theorem \ref{thm-comp-lpomega}  and Theorem \ref{thmmain} will be given in Section 4 and 5, respectively.

%%%%%%%%%%%%%%%%%%%%%%%%%%%%%%%%%%%%%%%%%%%%%%%%%%%%%%%%%%%%%%%%%%%
%%%%%%% Section 2. Preliminaries %%%%%%%%%%%%%%%%%%%%%%%%%%%%

\section{Preliminaries}
\subsection \noindent\textbf{Multilinear C-Z operator and multiple weights}

For $f_j\in \mathcal{S}(\mathbb{R}^{n})$, $1\leq j\leq m$, and
$x\not\in\bigcap_{j=1}^{m}\operatorname{supp}f_{j}$, the multilinear operator $T$ is defined by
\begin{equation*}\label{1}
T(f_{1},\cdots,f_{m})(x)
=\int_{(\mathbb{R}^{n})^{m}}K(x,y_{1},\dots,y_{m})f_{1}(y_{1})\cdots
f_{m}(y_{m})dy_{1}\dots dy_{m},
\end{equation*}where  the kernel $K(x,y_1,\dots,y_m)$ is a locally integrable function defined away from the diagonal
$x=y_1=\dots=y_m$ in $(\mathbb{R}^n)^{m+1}$,
satisfying, for some $\varepsilon,A_{\varepsilon}>0$,
\begin{enumerate}
\item[\emph{(i)}]$
|K(x,y_1,\dots,y_m)|\leq\frac{C}{(\sum\limits_{j=1}^m|x-y_j|)^{mn}};$
\item[\emph{(ii)}]
$ |K(x,y_1,\dots,y_i,\dots,y_m)-K(x,y_1,\dots,y_i',\dots,y_m)|
\leq\frac{C|y_i-y_i'|^{\varepsilon}}{(\sum^m_{j=1}|x-y_j|)^{mn+\varepsilon}}$\\
whenever $|y_i-y_i'|\leq\frac{1}{2}\max_{1\leq j \leq m}|x-y_j|$;
\item[\emph{(iii)}]
$|K(x,y_1,\dots,y_m)-K(x',y_1,\dots,y_m)|
\leq\frac{C|x-x'|^{\varepsilon}}{(\sum^m_{j=1}|x-y_j|)^{mn+\varepsilon}}$\\
whenever
$|x-x'|\leq\frac{1}{2}\sum_{j=1}^m|x-y_j|$.
\end{enumerate}

An $m-$linear (or quasi-linear) operator is called \emph{bounded at some point} if it is bounded from $L^{p_1}\times\dots\times L^{p_m}$ to $L^{p}$
for some $1\leq p_i\leq\infty$,
$\frac1p=\sum_{j=1}^m\frac1{p_m}$ and $p<\infty$.

Following the work of Grafakos and Torres \cite{GrT}, we call the above $T$ is a mutilinear Calder\'{o}n-Zygmund operator if it is bounded at some point. Grafakos and Torres \cite{GrT} showed that 
a mutilinear Calder\'{o}n-Zygmund operator is always bounded from
$L^{1}\times\dots\times L^{1}$ to $L^{\frac1m,\infty}$.
After introducing the multiple weights associated with the so-called new maximal functions, the authors in \cite{LOPTT} established a nice weighted theory for $T$. We first recall the definition of  $A_{\vec{p}}$ weight class. 
 \begin{dfn} \cite{LOPTT}. Let $1\leq p_1,\cdots,p_m<\infty,$ $p$ satisfies
$\frac{1}{p}=\frac{1}{p_1}+\cdots+\frac{1}{p_m}.$ For any $i=1,\cdots,m$, let $\omega_i$ be a weight, which is nonnegative and locally integrable function.  Given
$\vec{\omega}=(\omega_1,\cdots, \omega_m)$, set
$\nu_{\vec{\omega}}=\prod_{i=1}^m\omega_i^{p/p_i}$.
We say that $\vec{\omega}$ satisfies the $A_{\vec{p}}$ condition if
$$\sup_B \left(\frac{1}{|B|}\int_B \prod_{i=1}^m\omega_i^{p/p_i}\right)^{\frac{1}{p}}\prod_{i=1}^m\left( \frac{1}{|B|}\int_B \omega_i^{1-p_i^{'}}\right)^{\frac{1}{p_i^{'}}}< \infty,$$
when $p_i=1,$ $\left( \frac{1}{|B|}\int_B
\omega_i^{1-p_i^{'}}\right)^{{1}/{p_i^{'}}}$ is understood as
$(\inf\limits_B \omega_i)^{-1}.$
\end{dfn}
Given a nonempty set $A\subset\{1,\dots, m\}$, define
\begin{align*}
\mathcal{M}_A(\vec f)(x)=\sup_{Q\ni x}\prod_{j\in A}\frac{1}{|Q|}\int_{Q}|f_j(y_j)|dy_j
\end{align*}
and denote $\mathcal{M}_{\emptyset} (\vec f)(x)=1.$
The related results in \cite{LOPTT} can be summarized as follows.
\begin{thm}\cite{LOPTT}\label{thm-loptt-1} Let $A=\{1,2,\cdots,m\}$, $1< p_{i}<
\infty$,
$\frac{1}{p}=\frac{1}{p_{1}}+\cdots+\frac{1}{p_{m}}$, and
$\vec{\omega}$ satisfy the $A_{\vec{p}}$ condition. Then
\begin{equation*}
{\big\|\mathcal{M}_A(\vec f)\big\|}_{L^p({\nu_{\vec{\omega} }})} \leqslant C
\prod_{i = 1}^m{\big\|f_i\big\|}_{L^{p_i}(\omega_i)}.
\end{equation*}
\end{thm}

\begin{thm}\cite{LOPTT}\label{thm-loptt-2} Let $T$ be an
$m$-linear Calder\'{o}n-Zygmund operator and
$\vec{\omega}$ satisfy the $A_{\vec{p}}$ condition, $1< p_{i}<
\infty$. Then
\begin{equation*}
{\big\|T(\vec{f})(x)\big\|}_{L^p({\nu_{\vec{\omega} }})} \leqslant C
\prod_{i = 1}^m{\big\|f_i\big\|}_{L^{p_i}(\omega_i)}.
\end{equation*}
\end{thm}

\subsection \noindent\textbf{Generalized commutators}

The following three kinds of
commutators were firstly introduced and studied in\cite{PeTr}, \cite{LOPTT}  and \cite{PPTT}, respectively.
\begin{equation}\label{CT1}
T_{\vec{b}}(f)(x)=\int_{\mathbb{R}^n}[\prod_{j=1}^m(b_j(x)-b_j(y))]K(x,y)f(y)dy,
\end{equation}

\begin{equation}\label{CT2}
T_{\vec{b}}(\vec{f})(x)=\sum_{i=1}^m\big(b_iT(\vec{f})(x)
-T(f_1,\cdots,b_if_i,\cdots,f_m)(x)\big),
\end{equation}

\begin{equation}\label{CT3}
T_{\prod
b}(\vec{f})(x)=\int_{\mathbb{R}^{nm}}\prod_{j=1}^m(b_j(x)-b_j(y_j))
K(x,y_1,\dots,y_m)dy_1\dots dy_m.
\end{equation}

The other main results obtained by these papers are that all these commutators enjoy
a natural weighted strong and weighted endpoint
boundedness.
Although these commutators are of course different with each other and subsequently the proofs and the results are independent there. However,  one can always see that their proofs are all in similar patterns.
In \cite{XuY}, we introduced a kind of generalized commutators which contains the three type of commutators mentioned above.
\begin{dfn}\cite{XuY}
Let $T$ be an $m$-linear Calder\'{o}n-Zygmund operator with kernel
$K$. Let $S$ be a finite subset of $Z^+\times \{1,\dots,m\}$. The
commutators of $T$ are defined by
\begin{equation}\aligned\label{GC}
T_{\vec{b},S}(\vec{f})(x)=\int_{\mathbb{R}^{nm}}\prod_{(i,j)\in
S}(b_i(x)-b_i(y_j)) K(x,y_{1},\cdots,y_{m})
\prod_{j=1}^mf_j(y_j)d\vec{y},
\endaligned
\end{equation}
for all $f_j\in \mathcal{S}$, $j=1,\dots,m$, and all
$x\not\in\bigcap_{j=1}^{m}\operatorname{supp}f_{j}$.
\end{dfn}

Not very surprisingly, this kind of commutators also enjoys the natural weighted strong and weighted endpoint
boundedness. For instance, it holds that
\begin{thm}\cite{XuY}\label{thm-XuY}
Let $\vec{\omega}\in A_{\vec{p}}$ with
$\frac{1}{p}=\sum_{j=1}^m\frac{1}{p_j}$, $1<p_j<\infty$,
 $j=1,\dots,m$. Then there exists a
 constant $C$ such that for any $f_j\in L^p(\omega_j)$, it holds that
\begin{equation*}
\|T_{\vec{b},S}\vec{f})\|_{L^p(v_{\vec{\omega}})}\leq
C\prod_{(i,j)\in S}\|b_i\|_{BMO}
\prod_{j=1}^m\|f_j\|_{L^{p_j}(\omega_j)}.
\end{equation*}
\end{thm}

\subsection\noindent\textbf{Multilinear square operators}

Recall that, when it comes to the linear theory, we have basically three typical square operators,
the Littlewood-Paley $g$-function, the Marcinkiewicz integral and the $g_{\lambda}^*$-function. Correspondingly, we may define three kinds of multilinear square operators. To begin with, we first introduce two kinds of kernels.
\begin{dfn}\label{def Marcinkiewicz kernel}
Let $K$ be a function defined on $\mathbb{R}^n\times \mathbb{R}^{mn}$ with $\supp K\subseteq
\mathcal{B}:=\{(x,y_1,\dots,y_m):\sum_{j=1}^m|x-y_j|^2\leq 1\}$. $K$ is called a multilinear Marcinkiewicz kernel if for some $0<\delta<mn$ and some positive constants $A$, $\gamma_0$, and $B_1$,
\begin{enumerate}
	\item[\emph{(a)}]$
|K(x,\vec{y})|\leq\frac{A}
{(\sum_{j=1}^{m}|x-y_{j}|)^{mn-\delta}};
$	\item[\emph{(b)}]$
|K(x,\vec{y})-K(x,y_{1},\dots,y_i',\dots,y_{m})|
\leq\frac{A|y_i-y_i'|^{\gamma_0}}
{(\sum_{j=1}^{m}|x-y_{j}|)^{mn-\delta+\gamma_0}},
$;
	\item[\emph{(c)}]$
|K(x,\vec{y})-K(x',y_{1},\dots,y_{m})|
\leq\frac{A|x-x'|^{\gamma_0}}
{(\sum_{j=1}^{m}|x-y_{j}|)^{mn-\delta+\gamma_0}},
$,\end{enumerate}
where (b) holds whenever $(x,y_1,\dots,y_m)\in \mathcal{B}$ and
$|y_i-y_i'|\leq\frac{1}{B_1}|x-y_i|$ for all $0\leq i\leq m$,
and (c) holds whenever $(x,y_1,\dots,y_m)\in \mathcal{B}$ and
$|x-x'|\leq\frac{1}{B_1}\max_{1\leq j \leq m}|x-y_{j}|$.
\end{dfn}

\begin{dfn}\label{def littlewoodpayley kernel}
Let $K(x,y_1,\dots,y_m)$ be a locally integrable function defined away from the diagonal
$x=y_1=\dots=y_m$ in $(\mathbb{R}^n)^{m+1}$.
$K$ is called a multilinear Littlewood-Paley kernel if for some positive constants $A$, $\gamma_0$, $\delta$, and $B_1$, it holds that 
\begin{enumerate}
	\item[\emph{(d)}]$
|K(x,\vec{y})|\leq\frac{A}
{(1+\sum_{j=1}^{m}|x-y_{j}|)^{mn+\delta}};
$\item[\emph{(e)}]$
|K(x,\vec{y})-K(x,y_{1},\dots,y_i',\dots,y_{m})|
\leq\frac{A|y_i-y_i'|^{\gamma_0}}
{(1+\sum_{j=1}^{m}|x-y_{j}|)^{mn+\delta+\gamma_0}}
$;
\item[\emph{(f)}]$
|K(x,\vec{y})-K(x',y_{1},\dots,y_{m})|
\leq\frac{A|x-x'|^{\gamma_0}}
{(1+\sum_{j=1}^{m}|x-y_{j}|)^{mn+\delta+\gamma_0}},
$\end{enumerate}
where (e) holds whenever
$|y_i-y_i'|\leq\frac{1}{B_1}|x-y_i|$ and for all $0\leq i\leq m$,
and (f) holds whenever
$|x-x'|\leq\frac{1}{B_1}\max\limits_{1\leq j \leq m}|x-y_{j}|$.
\end{dfn}

Given a kernel $K$, denote
$K_t(x,y_1,\dots,y_m)={t^{-mn}}K(\frac{x}{t},\frac{y_1}{t},\dots,\frac{y_m}{t})$.
Define the multilinear square function by
\begin{equation*} \label{T}
\mathcal{G}(\vec{f})(x)=\big(\int_{0}^{\infty}|\int_{(\mathbb{R}^{n})^m}K_t(x, y_1,\cdots,y_m)
\prod_{j=1}^mf_j(y_j)dy_1\dots dy_m|^2\frac{dt}{t}\big)^{1/2},
\end{equation*}
for any $\vec{f}=(f_{1},\cdots,f_{m})\in
\mathcal{S}(\mathbb{R}^n)\times\mathcal{S}(\mathbb{R}^n)\times\cdots\times\mathcal{S}(\mathbb{R}^n)$ and all $x\notin \bigcap_{j=1}^m \texttt{supp} f_j.$

Suppose that $\mathcal{G}$ is bounded at some point. $\mathcal{G}$ is called a multilinear Marcinkiewicz operator when $K$ is a  multilinear Marcinkiewicz kernel. $\mathcal{G}$ is called a multilinear Littlewood-Paley $g$-function when $K$ is a multilinear Littlewood-Paley kernel, 

Meanwile, the multilinear square $g_{\lambda}^*$-function associated with the above kernel $K$ is defined by
\begin{equation*}\aligned \label{T_k}
\mathcal{G}_{\lambda}^*(\vec{f})(x)&=\bigg(\iint_{\mathbb{R}^{n+1}_+}\big(\frac{t}{|x-z|+t}\big)^{n\lambda}
|\int_{\mathbb{R}^{nm}}K_t(z,\vec{y})
\prod_{j=1}^mf_j(y_j)d\vec{y}|^2\frac{dzdt}{t^{n+1}}\bigg)^{\frac12},
\endaligned
\end{equation*}
whenever $\vec{f}=(f_{1},\cdots,f_{m})\in
\mathcal{S}(\mathbb{R}^n)\times\mathcal{S}(\mathbb{R}^n)\times\cdots\times\mathcal{S}(\mathbb{R}^n)$ and $x\notin \bigcap_{j=1}^m \supp f_j$, with itself bounded at some point.

See \cite{ChXY} \cite{XPY} \cite{ShXY} respectively for the convolution type of the above three kinds of multilinear square operators, where endpoint estimates as well as the weighted boundedness, like Theorem \ref{thm-loptt-2}, were obtained. Although each of the proofs is complete and independent and somehow seems quite different, once again, we can acttually tackle these square operators in a unified manner \cite{XuY2}, by viewing all of them as vector-valued multilinear square operators. The following two lemmas are the crucial estimates.
\begin{lem}\cite{XuY2}\label{lem-XuY2-1}
When $K$ is either a multilinear Littlewood-Paley kernel or multilinear Marcinkiewicz kernel, there exists some positive constants $\gamma$, $A$, and $B$, such that
\begin{equation*}\aligned
\big(\int_{0}^{\infty}|K_t(x,\vec{y})|^2\frac{dt}{t}\big)^{\frac12}
\leq \frac{A}
{(\sum_{j=1}^{m}|x-y_{j}|)^{mn}},
\endaligned
\end{equation*}
\begin{equation*}\aligned
\big(\int_{0}^{\infty}|K_t(z,\vec{y})-K_t(x,\vec{y})|^2\frac{dt}{t}\big)^{\frac12}
\leq \frac{A|z-x|^{\gamma}}
{(\sum_{j=1}^{m}|x-y_{j}|)^{mn+\gamma}},
\endaligned
\end{equation*}
whenever $|z-x|\leq \frac1B\max_{j=1}^m\{|x-y_j|\}$;
and
\begin{equation*}\aligned
\big(\int_{0}^{\infty}
|K_t(x,\vec{y})-K_t(x,y_{1},&\dots,y_i',\dots,y_{m})|^2\frac{dt}{t}\big)^{\frac12}
\leq \frac{A|y_i-y_i'|^{\gamma}}
{(\sum_{j=1}^{m}|x-y_{j}|)^{mn+\gamma}}
\endaligned\label{condition 12}
\end{equation*}
for any $i\in\{1,\dots,m\}$, whenever $|y_i-y_i'|\leq \frac{|x-y_i|}B$.
\end{lem}

\begin{lem}\cite{XuY2}\label{lem-XuY2-2}
When $K$ is multilinear Littlewood-Paley kernel, there exists some positive constants $\gamma$, $A$, and $B$, such that
\begin{equation*}\aligned
\big(\iint_{\mathbb{R}^{n+1}_+}\big(\frac{t}{|x-z|+t}\big)^{n\lambda}
|K_t(z,\vec{y})|^2\frac{dzdt}{t^{n+1}}\big)^{\frac1{2}}
\leq \frac{A}
{(\sum_{j=1}^{m}|x-y_{j}|)^{mn}};
\endaligned
\end{equation*}
\begin{equation*}\aligned
\big(\iint_{\mathbb{R}^{n+1}_+}\big(\frac{t}{|z|+t}\big)^{n\lambda}
|K_t(x-z,\vec{y})-&
K_t(x'-z,\vec{y})|^2\frac{dzdt}{t^{n+1}}\big)^{\frac1{2}}
\leq \frac{A|x-x'|^{\gamma}}
{(\sum_{j=1}^{m}|x-y_{j}|)^{mn+\gamma}},
\endaligned
\end{equation*}
whenever $|x-x'|\leq \frac1B\max_{j=1}^m\{|x-y_j|\}$;
and if $|y_i-y_i'|\leq \frac{|x-y_i|}B$, it holds that
\begin{equation*}\aligned
\big(\iint_{\mathbb{R}^{n+1}_+}\big(\frac{t}{|x-z|+t}\big)^{n\lambda}
|K_t(z,\vec{y})-&
K_t(z,y_1,\dots,y_i',\dots,y_m)|^2\frac{dzdt}{t^{n+1}}\big)^{\frac1{2}}\\
&\quad\leq \frac{A|y_i-y_i'|^{\gamma}}
{(\sum\limits_{j=1}^{m}|x-y_{j}|)^{mn+\gamma}}.
\endaligned
\end{equation*}
\end{lem}
Now, we define the generalized commutators of multilinear square operators.
\begin{align*}
\mathcal{G}_{\vec{b},S}(\vec f)(x)
&=\big(\int_{0}^{\infty}|\int_{(\mathbb{R}^{n})^m}\prod_{(i,j)\in
S}(b_i(x)-b_i(y_j))K_t(x, \vec{y})\\&
\quad\times\prod_{j=1}^mf_j(y_j)dy_1\dots dy_m|^2\frac{dt}{t}\big)^{1/2},
\end{align*}
and
\begin{align*}
\mathcal{G}^*_{\lambda,\vec{b},S}(\vec f)(x)&=\bigg(\iint_{\mathbb{R}^{n+1}_+}\big(\frac{t}{|x-z|+t}\big)^{n\lambda}
|\int_{\mathbb{R}^{nm}}\prod_{(i,j)\in
S}(b_i(x)-b_i(y_j))K_t(z,\vec{y})\\&\quad\times
\prod_{j=1}^mf_j(y_j)d\vec{y}|^2\frac{dzdt}{t^{n+1}}\bigg)^{\frac12}.
\end{align*}

\section{vector-valued theory}

Let $B$ be a quasi-Banach space, $0<p<\infty$. For a $B$-valued strongly measurable function defined on $\R^n$, define
\begin{align*}
L^{p}_B&=\bigg\{f:\big(\int_{\mathbb{R}^n}\|f(x)\|^p_{B}dx\big)^{\frac1p}<\infty\bigg\}
=\bigg\{f:\|f\|_{L^p_{B}}<\infty\bigg\}.
\end{align*}

We can extend the multilinear Calder\'{o}n-Zygmund theory to the vector-valued case without much extra efforts.
Let $K(x,y_1,\dots,y_m)$ be a $B$-valued locally integrable function defined away from the diagonal
$x=y_1=\dots=y_m$ in $(\mathbb{R}^n)^{m+1}$.
We define the $B$-valued multilinear Calder\'{o}n-Zygmund operator $T$ in the way that
\begin{equation*}%\label{1}
T(f_{1},\cdots,f_{m})(x)
=\int_{(\mathbb{R}^{n})^{m}}K(x,y_{1},\dots,y_{m})f_{1}(y_{1})\cdots
f_{m}(y_{m})dy_{1}\dots dy_{m}
\end{equation*}
for $f_j\in \mathcal{S}(\mathbb{R}^{n})$, $1\leq j\leq m$, and
$x\not\in\bigcap_{j=1}^{m}\operatorname{supp}f_{j}$, with the kernel
satisfying, for some $\varepsilon,A_{\varepsilon}>0$,
\begin{enumerate}
\item[\emph{(i)}]$
\|K(x,y_1,\dots,y_m)\|_{B}\leq\frac{C}{(\sum\limits_{j=1}^m|x-y_j|)^{mn}};$
\item[\emph{(ii)}]
$\|K(x,y_1,\dots,y_i,\dots,y_m)-K(x,y_1,\dots,y_i',\dots,y_m)\|_{B}
\leq\frac{C|y_i-y_i'|^{\varepsilon}}{(\sum^m_{j=1}|x-y_j|)^{mn+\varepsilon}}$\\
whenever $|x-x'|\leq\frac{1}{2}\max_{1\leq j \leq m}|x-y_j|$;
\item[\emph{(iii)}]
$\|K(x,y_1,\dots,y_m)-K(x',y_1,\dots,y_m)\|_{B}
\leq\frac{C|x-x'|^{\varepsilon}}{(\sum^m_{j=1}|x-y_j|)^{mn+\varepsilon}}$\\
whenever
$|x-x'|\leq\frac{1}{2}\sum_{j=1}^m|x-y_j|$,
\end{enumerate}
and if $T$ is bounded at some point.

Similarly, such a $B$-valued operator $T$ is said to be \emph{bounded at some point} if it is bounded from $L^{p_1}\times\dots\times L^{p_m}$ to $L^{p}_B$
for some $1\leq p_i\leq\infty$,
$\frac1p=\sum_{j=1}^m\frac1{p_m}$ and $p<\infty$.

We may get the vector-valued version of the Theorem \ref{thm-loptt-2}. As this is almost a step by step copies of the original proof, except for just adding the norm $\|\cdot\|_B$ step by step, we omit the proof. One can see \cite{BeCP}, \cite{RuRT},  \cite{Ha1}, \cite{Ha2} for part of the ideas.
\begin{thm}\label{thm-bounds-vec} Let $T$ be a $B$-valued
$m$-linear Calder\'{o}n-Zygmund operator,
$\frac{1}{p}=\frac{1}{p_{1}}+\cdots+\frac{1}{p_{m}}$, and
$\vec{\omega}$ satisfy the $A_{\vec{p}}$ condition, $1< p_{i}<
\infty$. Then there exists a
 constant $C$ such that for any $f_j\in L^p(\omega_j)$, it holds that
\begin{equation*}
{\big\|T(\vec{f})(x)\big\|}_{L^p({\nu_{\vec{\omega} }}),B} \leqslant C
\prod_{i = 1}^m{\big\|f_i\big\|}_{L^{p_i}(\omega_i)}.
\end{equation*}
\end{thm}

Now we may define the generalized commutators of the vector-valued
multilinear Calder\'{o}n-Zygmund operators. Once again,
this generalized commutator also satisfies the natural weighted strong and weighted endpoint
boundedness, and we omit the proof.
\begin{thm}\label{thm-bounds-vec-com}
Let $\vec{\omega}\in A_{\vec{p}}$ with
$\frac{1}{p}=\sum_{j=1}^m\frac{1}{p_j}$ with $1<p_j<\infty$,
 $j=1,\dots,m$. Then there exists
 constant $C$ such that for any $f_j\in L^p(\omega_j)$, it holds that
\begin{equation*}
\|T_{\vec{b},S}\vec{f})\|_{L^p(v_{\vec{\omega}}),B}\leq
C\prod_{(i,j)\in S}\|b_i\|_{BMO}
\prod_{j=1}^m\|f_j\|_{L^{p_j}(\omega_j)}.
\end{equation*}
\end{thm}

\begin{rem}
The boundedness in the above two theorems can both be extended to the endpoint cases, just as they hold for scalar-valued multilinear operators, and the proofs will be mostly copies of scalar-valued ones. We omit them.
\end{rem}
For functions $f_1(t)$ defined on $\R^+$ and functions $f_2(t,z)$ defined on $\R^{n+1}_+$, define their norm respectively by
\begin{align*}
\|f_1\|_{H_1}=\big(\int_{0}^{\infty}|f_1(t)|^2\frac{dt}{t}\big)^{\frac12},
\quad\|f_2\|_{H_2}=\big(\iint_{\mathbb{R}^{n+1}_+}\big(\frac{t}{|x-z|+t}\big)^{n\lambda}
|f_2(t,z)|^2\frac{dzdt}{t^{n+1}}\big)^{\frac1{2}}.
\end{align*}

Lemma \ref{lem-XuY2-1} and Lemma \ref{lem-XuY2-2} are thus leading to the following facts:
\label{thm-XuY2}
A multilinear Marcinkiewicz operator $\mu_{\Omega}$ or a multilinear Littlewood-Paley operator $g$ is a $H_1$-valued multilinear Calder\'{o}n-Zygmund operator. And a multilinear square function $g_{\lambda}^*$ is a $H_2$-valued multilinear Calder\'{o}n-Zygmund operator.

Combining Theorem \ref{thm-XuY} with the above facts, we have

\begin{cor}
Assume $\mathcal{T}$ be any one of the following three multilinear operators $\mu_{\Omega}$, $g$ or $g_{\lambda}^*$. Let $\mathcal{T}_{\vec b,S}$ be its generalized commutator defined similarly as in \ref{GC}, $1\leq p_j<\infty$,$j=1,\dots, m$,$\frac{1}{p}=\sum_{j=1}^m\frac{1}{p_j}$,ÇÒ$\omega_i\in A_{p_i}$.
Then there exists constant $C>0$, such that for any $b_i\in BMO$,$f_j\in C^{\infty}_c$, it holds that
\begin{equation*}
\|\mathcal{T}_{\vec{b},S}\vec{f})\|_{L^p(v_{\vec{\omega}})}\leq
C\prod_{(i,j)\in S}\|b_i\|_{BMO}
\prod_{j=1}^m\|f_j\|_{L^{p_j} (\omega_j)}.
\end{equation*}
\end{cor}

\begin{rem}
The above result is new not only because it considered the generalized commutators,
or the weighted case, but also for the facts that, there is no literature about compactness of commutators of multilinear square operators. The linear case is indeed known \cite{ChD}.
\end{rem}

\begin{rem}
We can obtain the weighted endpoint boundedness for the generalized commutators of multilinear square operators, just as Remark 1 implied.
\end{rem}
\section{proof of Theorem \ref{thm-comp-lpomega}}
We need the following Lemma.
\begin{lem}\label{lem:compactset}
Let $1< p<\infty$. Let $w$ be a weight on $\Rn$ such that
$w^{-p'/p}=w^{-1/(p-1)}$ is also a weight on $\Rn$. 
%($\essi_{|x|\le A}>0$ for every $A>0$ when $p=1$). 
Let $G$ be a subset of $L^p(w)$.
Then $G$ is relatively compact in $L^p(w)$
if it satisfies the following three conditions:
\begin{enumerate}[{\rm(i)}]
\item %(i)
There exists $K>0$ such that $\|f\|_{L^p(w)}\le K$ for all $f\in G$;
\item %(ii)
For any $\varepsilon>0$ there exists $A>0$ such that
\begin{equation*}
\biggl(\int_{|x|>A}|f(x)|^pw(x)dx\biggr)^{1/p}<\varepsilon\ \text{ for any }
f\in G;
\end{equation*}
\item %(iii)
For any $\varepsilon>0$ there exists $\delta>0$ such that
\begin{equation*}
\biggl(\int_{\Rn}|f(x+u)-f(x)|^pw(x)dx\biggr)^{1/p}<\varepsilon\
\text{ for any }f\in G \text{ and }|u|<\delta.
\end{equation*}
\end{enumerate}
\end{lem}
%%%%%%%%%%%%%%%%% end Lemma Compact set in L^p %%%%%%%%%%%%%%%%%%%%
%%%%%%%%%% Remark 1 Compact set in L^p %%%%%%%%%%%%%%%%%%%%
\begin{rem}
Since $G'=\{fw^{1/p}; f\in G\}$ is a subset of $L^p(\Rn)$,
from the Fr\'echet-Kolmogorov theorem (cf. \cite{Yo}), it follows that $G$
 is relatively
compact in $L^p(w)$ if and only if it satisfies the following three conditions:
\begin{enumerate}[(i)]
\item %(i)
There exists $K>0$ such that $\|f\|_{L^p(w)}\le K$ for all $f\in G$;
\item %(ii)
For any $\varepsilon>0$, there exists $A>0$ such that
\begin{equation*}
\biggl(\int_{|x|>A}|f(x)|^pw(x)dx\biggr)^{1/p}<\varepsilon, \ \text{ for any }
f\in G;
\end{equation*}
\item[{$\rm(iii')$}]
For any $\varepsilon>0$, there exists $\delta>0$ such that
\begin{equation*}
\biggl(\int_{\Rn}|f(x+u)w(x+u)^{1/p}-f(x)w(x)^{1/p}|^pdx\biggr)^{1/p}
<\varepsilon,\ \text{ for any }f\in G\text{ and }|u|<\delta.
\end{equation*}
\end{enumerate}
\end{rem}
%%%%%%%%%% end Remark 1 Compact set in L^p %%%%%%%%%%%%%%%%%%%%
%%%%%%%%%% Remark 2 Compact set in L^p %%%%%%%%%%%%%%%%%%%%
\begin{rem}
If $w\in A_p(\Rn)$ $1\le p<\infty)$, then $w$ satisfies the conditions in Lemma
\ref{lem:compactset}.
\end{rem}
%%%%%%%%%% end Remark 2 Compact set in L^p %%%%%%%%%%%%%%%%%%%%

%%%%%%%%%%%%%%%%% Proof of Lemma Compact set in L^p %%%%%%%%%%%%%%%%%%%%
\begin{proof}
%We shall prove only for the case $p>1$. 
%The case $p=1$ can be proved with a minor change. \par
For any measurable set $E\subset\Rn$ and a locally integrable function $f$ on
 $\Rn$, we set $f_E=|E|^{-1}\int_E f(x)dx$. Then, for $B=B(x,t)$ we have
\begin{align*}
\|f-f_B\|_{L^p(w)}
&=\biggl(\int_\Rn \Bigl|\frac{1}{|B(x,t)|}\int_{B(x,t)}(f(x)-f(y))dy\Bigr|^p
w(x)dx\biggr)^{\frac1p}
\\
&=\biggl(\int_\Rn \Bigl|\frac{1}{|B(0,t)|}\int_{B(0,t)}(f(x)-f(y+x))dy\Bigr|^p
w(x)dx\biggr)^{\frac1p},
\end{align*}
and so by Minkowski's inequality, we have 
\begin{equation*}
\|f-f_B\|_{L^p(w)}
\le \frac{1}{|B(0,t)|}\int_{B(0,t)}\biggl(\int_\Rn |(f(x)-f(y+x))|^p
w(x)dx\biggr)^{\frac1p}dy.
\end{equation*}
Thus for any fixed $\varepsilon>0$,
choosing $\delta>0$ in the assumption (iii) in Lemma \ref{lem:compactset},
we see that for $0<t<\delta$, it holds that
\begin{equation}\label{eq:compactset-1}
\|f-f_{B(\,\cdot\,,t)}\|_{L^p(w)}<\varepsilon \ \text{ for any }f\in G.
\end{equation}
We fix such a $t>0$.
We next estimate $f_B$ and $f_B-f_{B'}$.
\begin{equation}\label{eq:compactset-2}
|f_{B(x,t)}|\le
\biggl(\frac{1}{|B(x,t)|}\int_{B(x,t)}|f(y)|^pw(y)dy\biggr)^{\frac1p}
\biggl(\frac{1}{|B(x,t)|}\int_{B(x,t)}w(y)^{-\frac{p'}{p}}dy
\biggr)^{\frac1{p'}}.
\end{equation}
For $f_B-f_{B'}$, we get
\begin{align}
|f_{B(x,t)}-f_{B(y,t)}|\label{eq:compactset-3}
&=\Bigl|\frac{1}{|B(y,t)|}\int_{B(y,t)}f(u-y+x)du
-\frac{1}{|B(y,t)|}\int_{B(y,t)}f(u)du\Bigr| \notag
\\
&=\Bigl|\frac{1}{|B(y,t)|}\int_{B(y,t)}(f(u-y+x)-f(u))du\Bigr| \notag
\\
&\le \biggl(\frac{1}{|B(y,t)|}\int_{B(y,t)}|f(u-y+x)-f(u)|^pw(u)du
\biggr)^{\frac1p}\\&\quad\times
\biggl(\frac{1}{|B(y,t)|}\int_{B(y,t)}w(u)^{-\frac{p'}{p}}du
\biggr)^{\frac1{p'}}. \notag
\end{align}
Now choose $A>0$ in (ii).
 Since we have assumed that $w^{-p'/p}$ is
a weight on $\Rn$, we see that there exists $c_0>0$ such that
$\int_{B(x,t)}w(u)^{-\frac{p'}{p}}du\le c_0$ for all $|x|\le A$.
Hence, $\{f_{B(x,t)}\}_{f\in G}$ is equi-bounded and equi-continuous on the
closed ball $B(0,A)$. So, by Ascoli-Arzel\'a theorem, it is relatively compact
and so totally bounded in $C(B(0,A))$. Thus, there exist a finite number of
$f_1,f_2,\dots, f_k\in G$ such that
\begin{equation*}
\inf_{j}\sup_{|x|\le A}|f_{B(x,t)}-f_{j,B(x,t)}|<\varepsilon/w(B(0,A))^{1/p}
 \ \text{ for all }f\in G.
\end{equation*}
It follows that for $f\in G$ there exists $1\le j\le k$ such that
\begin{equation}\label{eq:compactset-4}
\sup_{|x|\le A}|f_{B(x,t)}-f_{j,B(x,t)}|<\varepsilon/w(B(0,A))^{1/p}.
\end{equation}
For these $f,f_j$ we have
\begin{align*}
\|f-f_j\|_{L^p(w)}
&\le
\|(f-f_j)\chi_{|x|\le A}\|_{L^p(w)}+\|(f-f_j)\chi_{|x|> A}\|_{L^p(w)}
\\
&\le
\|(f-f_{B(x,t)})\chi_{|x|\le A}\|_{L^p(w)}
+\|(f_{B(x,t)}-f_{j,B(x,t)})\chi_{|x|\le A}\|_{L^p(w)}
\\
&\hspace{0.5cm}+\|(f_{j,B(x,t)}-f_j)\chi_{|x|\le A}\|_{L^p(w)}
+\|f\chi_{|x|>A}\|_{L^p(w)}+\|f_j\chi_{|x|>A}\|_{L^p(w)}.
\end{align*}
Hence by (iii), \eqref{eq:compactset-1} and \eqref{eq:compactset-4} we obtain
\begin{equation*}
\|f-f_j\|_{L^p(w)}\le 5\varepsilon,
\end{equation*}
which means that $G$ is totally bounded and hence relatively compact in
$L^p(w)$.
This completes the proof.
\end{proof}

We now restate Theorem \ref{thm-comp-lpomega} 
 in the following form:
 \begin{lem}\label{lem:compactset-2}
Let $w$ be a weight on $\Rn$. Assume that $w^{-1/(p_0-1)}$ is also a weight
on $\Rn$ for some $p_0>1$. Let $0<p<\infty$ and $F$ be a subset in $L^p(w)$.
Then $F$ is relatvely compact in $L^p(w)$ if the following three
conditions are satisfied:
\begin{enumerate}[{\rm(i)}]
\item %(i)
There exists $K>0$ such that $\|f\|_{L^p(w)}\le K$ for all $f\in F$;
\item %(ii)

For any $\varepsilon>0$ there exists $A>0$ such that
\begin{equation*}
\biggl(\int_{|x|>A}|f(x)|^pw(x)dx\biggr)^{1/p}<\varepsilon\ \text{ for any }
f\in F;
\end{equation*}
\item %(iii)
For any $\varepsilon>0$ there exists $\delta>0$ such that
\begin{equation*}
\biggl(\int_{\Rn}|f(x+u)-f(x)|^pw(x)dx\biggr)^{1/p}<\varepsilon\
\text{ for any }f\in F \text{ and }|u|<\delta.
\end{equation*}
\end{enumerate}
We note that if $0<p<1$, the metric is defined by $\|\,\cdot\,\|_{L^p(w)}^p$.
\end{lem}
%%%%%%%%%% Lemma 2 Compact set in L^p %%%%%%%%%%%%%%%%%%%%%%%%%%
%%%%%%%%%% end Lemma 2 Compact set in L^p %%%%%%%%%%%%%%%%%%%%%%%%%%

Using the idea in Tsuji \cite{Ts}, we may demonstrate Lemma \ref{lem:compactset-2} now.\begin{proof}
If $p\ge p_0$, then $w^{-1/(p_0-1)}\in L^1_{\mathrm{loc}}(\Rn)$ implies 
$w^{-1/(p-1)}\in L^1_{\mathrm{loc}}(\Rn)$, and so 
we get the conclusion by Lemma \ref{lem:compactset}. So, we
assume $p<p_0$ and set $0<a=p/p_0<1$. 
Without loss of generality, we assume
every $f\in F$ is nonnegative function. By an elementary calculation
(see \cite{Ts}) it holds that
\begin{equation}\label{eq:compactset-5}
|s^a-t^a|\le |s-t|^a\ \text{ for }s,t>0,
\end{equation}
and
\begin{equation}\label{eq:compactset-6}
|s-t|^a\le \frac{1}{a}\Bigl(\frac{s+t}{|s-t|}\Bigr)^{1-a}|s^a-t^a|
\ \text{ for }s,t>0.
\end{equation}

Using \eqref{eq:compactset-5}, we have
\begin{align*}
&\int_{\Rn}|f(x+u)^a-f(x)^a|^{p_0}w(x)dx
\\
&\hspace{1cm} \le
\int_{\Rn}|f(x+u)-f(x)|^{ap_0}w(x)dx
=\int_{\Rn}|f(x+u)-f(x)|^{p}w(x)dx,
\end{align*}
which implies that $F^a:=\{f^a; f\in F\}$ satisfies the condition (iii) in
Lemma \ref{lem:compactset} for $p_0$.
Also, $F^a$ satisfies the conditions (i) and (ii) in
Lemma \ref{lem:compactset} for $p_0$. Hence, by Lemma \ref{lem:compactset},
we see that $F^a$ is relatively compact in $L^{p_0}(w)$.
Now let $\{f_j\}$ be a sequence of functions in $F$. Since $F^a$ is relatively
compact in $L^{p_0}(w)$, there exists a Cauchy subsequence of $\{f_j^a\}$,
which we denote again by $\{f_j^a\}$ for simplicity.
Then for any $\varepsilon>0$, there exists an integer $N$ such that for
$i,j\ge N$, it follows that
\begin{equation}\label{eq:compactset-7}
\int_\Rn |f_i^a(x)-f_j^a(x)|^{p_0}w(x)dx<\varepsilon ^{p_0}.
\end{equation}
Let $E_\varepsilon $ be the set in $\Rn$ such that
\begin{equation*}
\frac{f_i(x)+f_j(x)}{|f_i(x)-f_j(x)|}\le \frac{1}{\varepsilon}.
\end{equation*}
Then, noting $ap_0=p$ and using \eqref{eq:compactset-6},
\eqref{eq:compactset-7}, we have
\begin{align*}
\int_{E_\varepsilon}|f_i(x)-f_j(x)|^{p}w(x)dx
&\le a^{-p_0}\varepsilon^{(a-1)p_0}
\int_{E_\varepsilon}|f_i^a(x)-f_j^a(x)|^{p_0}w(x)dx
\\
&\le a^{-p_0}\varepsilon^{(a-1)p_0}\varepsilon^{p_0}
=a^{-p_0}\varepsilon^{p}.
\end{align*}
On $E_\varepsilon^c$, by \eqref{eq:compactset-6} and (i) we have
\begin{align*}
\int_{E_\varepsilon^c}|f_i(x)-f_j(x)|^{p}w(x)dx
&\le
\int_{E_\varepsilon^c}|\varepsilon(f_i(x)+f_j(x))|^{p}w(x)dx
\\
&\le \varepsilon^p
\Bigl(\int_{E_\varepsilon^c}|f_i(x)|^{p}w(x)dx
+\int_{E_\varepsilon^c}|f_j(x)|^{p}w(x)dx\Bigr)\\&
\le 2K^p\varepsilon^p.
\end{align*}
By the above two estimates, it follows that $\{f_j\}$ is a Cauchy sequence in
$F\subset L^p(w)$. Thus $F$ is relatively compact in $L^p(w)$,
which completes the proof.
\end{proof}
%%%%%%%%%% end Proof of Lemma 2 Compact set in L^p %%%%%%%%%%%%%%%%%%%%%%%%%%
\begin{rem}
If $w\in A_\infty(\Rn)$, $w$ satisfies the assumption in
Lemma \ref{lem:compactset-2} for some $1<p_0<\infty$.
\end{rem}
Finally in this section we present a counter-example for the necessity of the 
condition (iii) in the above results.
Let $1<p_0<p<\infty$ and $1/p<\alpha<p_0/p$. Set
\begin{align*}
w(x)=|x|^{p_0-1} \quad\text{ and }\quad 
f(x)=|x|^{-\alpha}\chi_{\{|x|\le1\}}.
\end{align*}
Then we get $p_0-1-p\alpha>-1$ and $p\alpha>1$, and hence
\begin{align*}
w\in A_p(\R),\quad f\in L^p(w), \quad\text{ but}\quad
f(\cdot+h)\notin L^p(w), \forall h\not=0.
\end{align*}
So, letting $\F=\{f\}$, we see that $F$ is a compact set in $L^p(w)$.
But $\F$ does not safisfy (iii).

\section{Proof of Theorem \ref{thmmain}}
Several lemmas will be needed to prove results for square operators.

\begin{lem}\label{yl2-ch7}
There exists constant $C>0$, such that for any $\delta>0$, $f_j\in C^{\infty}_c$, $j=1,\dots, m$, and any $a\in\{1,\dots, m\}$, it holds that
$$
\int_{\sum_{j=1}^m|x-y_j|\leq \delta}\frac{\prod_{j=1}^m|f_j(y_j)|}{\big(\sum_{j=1}^m|x-y_j|\big)^{nm-1}}d\vec y\leq C\delta\mathcal{M} (\vec f)(x);
$$
$$
\int_{\sum_{j=1}^m|x-y_j|\geq \delta}\frac{\prod_{j=1}^m|f_j(y_j)|}{\big(\sum_{j=1}^m|x-y_j|\big)^{nm+1}}d\vec y\leq \frac C\delta\mathcal{M} (\vec f)(x).
$$
\end{lem}

The above estimate when $m=2$ is shown in \cite{BeDMT}. Its idea however can be applied to general $m\in\N$. Once again we omit the proof.

%%%%%%%%%% Lemma 5.2 %%%%%%%%%%%%%%%%%%%%%%%%%%5
\begin{lem}\label{yl3-ch7}
For any nonempty set $A\subset\{1,\dots, m\}$, any constant $N>0$, 
there exists a constant $C>0$, such that for any $f_j\in C^{\infty}_c$, 
$j=1,\dots, m$, with $\supp f_k\subset B(0,N)$, $k\in A$, and $|x|>2N$
\begin{align*}
\int_{\R^{nm}}\frac{\prod_{j=1}^m|f_j(y_j)|}
{\big(\sum_{j=1}^m|x-y_j|\big)^{nm}}d\vec y
\leq \frac{C}{|x|^{n|A|}}\prod_{j\in A}\|f_j\|_{L^1}\mathcal{M}_{A^c}
 (\vec f)(x).
\end{align*}
\end{lem}
%%%%%%%%%%%% end Lemma 5.2 %%%%%%%%%%%%%%%%%%%%%%%%%%%%55
%\noindent{\textbf{Proof of Lemma \ref{yl3-ch7}}}
\begin{proof}
Since $\supp f_k\subset B(0,N)$, $k\in A$, $|x|\geq 2N$, then
\begin{align*}
\int_{\R^{nm}}\frac{\prod_{j=1}^m|f_j(y_j)|}
{\big(\sum_{j=1}^m|x-y_j|\big)^{nm}}d\vec y\leq C\prod_{j\in A}\|f_j\|_{L^1}\int_{\R^{n|A^c|}}\frac{\prod_{j\in A^c}|f_j(y_j)|}
{\big(|x|+\sum_{j\in A^c}|x-y_j|\big)^{nm}}\prod_{j\in A^c}dy_j.
\end{align*}
So it suffices to show that
\begin{align*}
I:=\int_{\R^{n|A^c|}}\frac{\prod_{j\in A^c}|f_j(y_j)|}
{\big(|x|+\sum_{j\in A^c}|x-y_j|\big)^{nm}}\prod_{j\in A^c}dy_j
\leq \frac{C}{|x|^{n|A|}}\mathcal{M}_{A^c} (\vec f)(x).
\end{align*}
Decomposing $\R^{n|A^c|}$, we get
\begin{align*}
I&\leq \frac{C}{|x|^{mn}}
\int_{\sum_{j\in A^c}|x-y_j|\leq |x|}\prod_{j\in A^c}|f_j(y_j)|dy_j\\
&\quad+\sum_{k=0}^{\infty}\frac{C}{(2^k|x|)^{mn}}\int_{2^k|x|\leq\sum_{j\in A^c}|x-y_j|\leq 2^{k+1}|x|}\prod_{j\in A^c}|f_j(y_j)|\prod_{j\in A^c}dy_j\\
&\leq \frac{C}{|x|^{mn}}\prod_{j\in A^c}
\int_{B(x,|x|)}|f_j(y_j)|dy_j+\sum_{k=0}^{\infty}\frac{C}{(2^k|x|)^{mn}}\prod_{j\in A^c}\int_{B(x,2^{k+1}|x|)}|f_j(y_j)|dy_j\\
&\leq \frac{C}{|x|^{n|A^c|}}\prod_{j\in A^c}\frac1{|x|^n}
\int_{B(x,|x|)}|f_j(y_j)|dy_j\\
&\quad+\frac{C}{|x|^{n|A^c|}}\sum_{k=0}^{\infty}\frac1{2^{k|A|}}
\prod_{j\in A^c}\frac{1}{(2^k|x|)^n}\int_{B(x,2^{k+1}|x|)}|f_j(y_j)|dy_j.
\end{align*}
By the definition of $\mathcal{M}_{A^c} (\vec f)$, one can get Lemma \ref{yl3-ch7}.
\end{proof}
%%%%%%%%%%%%% end Proof of Lemma 5.2 %%%%%%%%%%%%%%%%%%%%%%%
%\par\smallskip
Now take $u\in C^{\infty} ([0,\infty))$ and 
$\chi_{[1,\infty)}\leq u(t)\leq \chi_{[1/2,\infty)}$. It follows that 
$|1-u(t)|\leq \chi_{(0,1)}(t)$ when $t>0$.
Denote
\begin{align*}
&T_{\vec b,S,\delta} (f_1,\dots,f_m)(x)
\\
&=\int_{\mathbb{R}^n}u(\frac{\sum_{i=1}^m|x-y_i|}{\delta})
\prod_{(i,j)\in S} (b_i(x)-b_i(y_j))K(x,y_1,\dots,y_m)
\prod_{j=1}^mf_j(y_j)dy_1\dots dy_m.
\end{align*}

%%%%%%%%%%%%%% Lemma 5.3 %%%%%%%%%%%%%%%%%%%%%%%%%%%
\begin{lem} \label{5.3}
	Let $B$ be a Banach space, $T$ be a $B$-valued $m$-linear Calder\'{o}n-Zygmund 
	operator, $1< p_{i}<\infty$, and $\frac{1}{p}=\frac{1}{p_{1}}
	+\cdots+\frac{1}{p_{m}}$. 
	Let $S$ be a finite set in $\mathbb Z^+\times \{1,\dots,m\}$, 
	$A=\{j\in\{1,\dots,m\}: (i,j)\in S\text{ for some }i\in\mathbb Z^+\}$ and 
	$S_j=\{i:(i,j)\in S\}$ for $1\le j\le m$.  
	Then, 
	if $\vec{\omega}\in A_{\vec{p}}$, $\vec\omega_{A^c}\in A_{\vec p,A^c}$, 
	$\nu_{\vec\omega,A}\in A_{p_A|A|}$, and $b_i\in C_c^\infty$ for 
	$i\in \bigcup_{j=1}^m S_j$,
	then $T_{\vec b,S,\delta}$ converges to 
	$T_{\vec b,S}$ in 
	$B(L^{p_1} (\omega_1)\times\dots\times L^{p_m} (\omega_m),
	L^p_B(\nu_{\vec\omega}))$ when $\delta\rightarrow 0$.
\end{lem}
%%%%%%%%%%%%% end Lemma 5.3 %%%%%%%%%%%%%%%%%%%%%%%%%%%%%%%
\begin{proof}
	
	It follows that
	\begin{align*}
	&\|T_{\vec b,S} (\vec f)(x)-T_{\vec b,S,\delta} (\vec f)(x)\|_B\\
	&\leq\int_{\sum_{j=1}^m|x-y_j|\leq \delta}\prod_{(i,j)\in S} |b_i(x)-b_i(y_j)|\|K(x,y_1,\dots,y_m)\|_B|\prod_{j=1}^mf_j(y_j)|dy_1\dots
	dy_m.
	\end{align*}
	Let $(i_0, j_0)\in S$, then
	\begin{align*}
	&\|T_{\vec b,S} (\vec f)(x)-T_{\vec b,S,\delta} (\vec f)(x)\|_B
	\leq C\|\nabla b_{i_0}\|_{L^{\infty}}
	\int_{\sum_{j=1}^m|x-y_j|\leq \delta}\frac{\prod_{j=1}^m|f_j(y_j)|}{\big(\sum_{j=1}^m|x-y_j|\big)^{nm-1}}d\vec y.
	\end{align*}
	By Lemma \ref{yl2-ch7}, one may obtain that
	\begin{align}\label{deltaboud}
	&\|T_{\vec b,S} (\vec f)(x)-T_{\vec b,S,\delta} (\vec f)(x)\|_B\leq C\|\nabla b_{i_0}\|_{L^{\infty}}\delta\mathcal{M} (\vec f)(x).
	\end{align}
	By Theorem \ref{thm-loptt-1}, Lemma \ref{5.3} is proved.
\end{proof}
%%%%%%%%%% end Proof of Lemma 5.2 %%%%%%%%%%%%%%%%%%%%%%%%%%%%%%%%%%%%%%%%%%%%
Note that if $T_n\in K(X_1\times\dots\times X_m,Y)$, $n=1,\dots$, and $T_n$ 
converge to some $T\in B(X_1\times\dots\times X_m,Y)$ in 
$B(X_1\times\dots\times X_m,Y)$, then
$T\in K(X_1\times\dots\times X_m,Y)$. 
In order to prove Theorem \ref{thmmain}, we may assume $b_i\in C_c^\infty$ for 
$i\in \bigcup_{j=1}^m S_j$ by Theorem \ref{thm-bounds-vec-com} and 
the density of $C_c^\infty$ in $CMO$. 
Therefore, we only need to prove under 
the conditions of Theorem \ref{thmmain}, for any 
$\delta>0$, $T_{\vec b,S,\delta}$ is a compact operator from 
$L^{p_1} (\omega_1)\times\dots\times L^{p_m} (\omega_m)$ to 
$L^p_B(\nu_{\vec\omega})$.

By Lemma \ref{5.3} and Theorem \ref{thm-comp-lpomega}, it suffices to show the following theorem:
%%%%%%%%%%%% Theorem 5.1 %%%%%%%%%%%%%%%%%%%%%%%%%%%%
\begin{thm}\label{dl5-ch7} Let $T_{\vec b,S,\delta}$ be defined as the same in Lemma \ref{5.3}. Then
\begin{enumerate}
\item[\emph{(i)}] There exists constant $C>0$, such that for any $f_i\in C^{\infty}_c$, $i=1,\dots, m$,
\begin{align}\label{inthm-bounds-Tdeltab}
\|T_{\vec b,S,\delta} (\vec f)\|_{L^p_B(\nu_{\vec\omega})}\leq C\prod_{j=1}^m\|f_j\|^p_{L^p(\omega_i)};
\end{align}
\item[\emph{(ii)}]For any $\epsilon>0$, there exists constant $N>0$, such that for any $f_i\in C^{\infty}_c$, $i=1,\dots, m$,
$$
\int_{|x|\geq N}\|T_{\vec b,S,\delta} (\vec f)(x)\|_B^p\nu_{\vec\omega}dx\leq 
\epsilon\prod_{j=1}^m\|f_j\|_{L^p(\omega_i)};
$$
\item[\emph{(iii)}]For any $\epsilon>0$, there exists constant $N>0$, 
such that for any $f_i\in C^{\infty}_c$, $i=1,\dots, m$, $|t|<\delta$,
$$
\int_{\mathbb{R}^n}\|T_{\vec b,S,\delta} (\vec f)(x+t)-T_{\vec b,S,\delta} 
(\vec f)(x)\|_B^p\nu_{\vec\omega}dx\leq \epsilon
\prod_{j=1}^m\|f_j\|^p_{L^p(\omega_i)}.
$$
\end{enumerate}
\end{thm}
%%%%%%%%%%%%%% end Theorem 5.1
\begin{proof}
(i) follows directly from \eqref{deltaboud}, Theorem \ref{thm-bounds-vec-com} and Theorem \ref{thm-loptt-1}. Now we prove (ii).
Since $b_i\in C^{\infty}_c(\R ^n)$ for any $i\in\bigcup_{j=1}^mS_j$,
there exists $N_0\in \mathbb{N}$ such that
$\supp b_i \in B(0, N_0)$ for any $i\in \bigcup_{j=1}^mS_j$.
So for $|x|\ge N\ge 2N_0,$ it holds that
\begin{align*}
\|T_{\vec b,S,\delta} (\vec f)(x)\|_B
&\leq C\prod_{(i,j)\in S}\|b_i\|_{L^{\infty}}
\int_{\R^{nm}}\frac{\prod_{j=1}^m|f_j(y_j)|}
{\big(\sum_{j=1}^m|x-y_j|\big)^{nm}}
\prod_{j\in A}\chi_{B(0,N_0)} (|y_j|)d\vec y.
\end{align*}
By Lemma \ref{yl3-ch7}, one obtains that
\begin{align*}
\|T_{\vec b,S,\delta} (\vec f)(x)\|_B
\leq C\prod_{(i,j)\in S}\|b_i\|_{L^{\infty}}\frac{1}{|x|^{n|A|}}
\prod_{j\in A}\|f_j\|_{L^1}\mathcal{M}_{A^c} (\vec f)(x).
\end{align*}
By the fact that $\frac1p=\frac1{p_A}+\frac{1}{p_{A^c}}$, $\nu_{\vec\omega}=\nu_{\vec\omega,A}^{\frac{p}{p_A}}
\nu_{\vec\omega,A^c}^{\frac{p}{p_{A^c}}}$, we have
\begin{align*}
&\big(\int_{|x|\geq N}\|T_{\vec b,S,\delta} (\vec f)(x)\|_B^p\nu_{\vec\omega}dx\big)^{\frac1p}\\
&\leq C\prod_{(i,j)\in S}\|b_i\|_{L^{\infty}}\prod_{j\in A}\|f_j\|_{L^1}\big(\int_{|x|\geq N}\frac{1}{|x|^{|A|p}}|\mathcal{M}_{A^c} (\vec f)(x)|^p\nu_{\vec\omega}dx\big)^{\frac1p}\\
&\leq C\prod_{j\in A}(\omega^{-\frac{p_j'}{p_j}}_j(B(0,N_0)))^{1/p_j'}\prod_{(i,j)\in S}\|b_i\|_{L^{\infty}}\prod_{j\in A}\|f_j\|_{L^{p_j} (\omega_j)}\\
&\quad\times\big(\int_{|x|\geq N}\frac{1}{|x|^{n|A|p_{A}}}\nu_{\vec\omega,A}dx\big)^{\frac1{p_A}}\|\mathcal{M}_{A^c} (\vec f)\|_{L^{p_{A^c}}} (\nu_{\vec\omega,{A^c}}).
\end{align*}
%Notice that 
From the assumptions 
$\vec\omega_{A^c}\in A_{\vec p,A^c}$, $\nu_{\vec\omega,A}\in A_{p_A|A|}$, 
it follows 
\begin{align*}
&\big(\int_{|x|\geq N}\|T_{\vec b,S,\delta} (\vec f)(x)\|_B^p\nu_{\vec\omega}dx\big)^{\frac1p}\\
&\leq C\prod_{j\in A}(\omega^{-\frac{p_j'}{p_j}}_j(B(0,N_0)))^{1/p_j'}\prod_{(i,j)\in S}\|b_i\|_{L^{\infty}}\prod_{j=1}^m\|f_j\|_{L^{p_j} (\omega_j)}\big(\int_{|x|\geq N}\frac{1}{|x|^{n|A|p_{A}}}\nu_{\vec\omega,A}dx\big)^{\frac1{p_A}}.
\end{align*}
Note for any $1<q<\infty$, any $\omega\in A_q$, it holds that
\begin{align*}
\int_{\mathbb{R}^n}\frac{\omega(x)}{(1+|x|)^{nq}}dx<\infty.
\end{align*}
and hence $\int_{|x|\geq N}\frac{1}{|x|^{n|A|p_{A}}}\nu_{\vec\omega,A}dx\rightarrow 0$ as $N\rightarrow \infty.$
Indeed, since $\nu_{\vec\omega,A}\in A_{p_A|A|}$, there exists $1<r_A<p_A|A|$ such that 
$\nu_{\vec\omega,A}\in A_{r_A}$. So we get

$$\aligned
\int_{|x|\ge N } \frac {\nu_{\vec\omega,A}(x)dx}{|x|^{np_A|A|}}&=
\sum_{k=0}^\infty \int_{2^kN\le |x|<2^{K+1}N}  \frac {\nu_{\vec\omega,A}(x)dx}{|x|^{np_A|A|}}\\
&\le \sum_{k=0}^\infty\frac 1{(2^kN)^{np_A|A|}} \int_{ |x|<2^{k+1}N}  {\nu_{\vec\omega,A}(x)dx}.
\endaligned$$
It is known that if $1\le p <\infty$ and $\omega\in A_p(\mathbb R^n)$, then 
$\omega(\lambda B)\le [\omega]_{A_p} \lambda^{np} \omega(B)$ for any $\lambda>0$ and any 
ball $B$.  (see, for example, Proposition 7.1.5 (9), page 504, Grafakos book , Classical Fourier Analysis.)
Using this, we  get
$$\aligned
\sum_{k=0}^\infty\frac 1{(2^kN)^{np_A|A|}} \int_{ |x|<2^{k+1}N}  {\nu_{\vec\omega,A}(x)dx}
&\le \sum_{k=0}^\infty \frac {[\nu_{\vec\omega,A}]_{A_{r_A}}}{(2^kN)^{np_A|A|}}(2^{k+1}N)^{nr_A}\nu_{\vec\omega,A}(B(0,1))\\&\le C\frac {1}{N^{np_A|A|-nr_A}}.
\endaligned$$

Thus, (ii) is obtained.

So, it suffices to show (iii). Note that
\begin{align*}
&\|T_{\vec b,S,\delta} (f_1,\dots,
 f_m)(x+t)-T_{\vec b,S,\delta,} (f_1,\dots,
 f_m)(x)\|_B\\
 &=\|\int_{\mathbb{R}^n}\big(\prod_{(i,j)\in S} (b_i(x+t)-b_i(y_j))K_{\delta} (x+t,y_1,\dots,y_m)-\\
&\quad\times\prod_{(i,j)\in S} (b_i(x)-b_i(y_j))K_{\delta} (x,y_1,\dots,y_m)\big)f_1(y_1)\dots f_m(y_m)d\vec y\|_B\\
&\leq I+II,\end{align*}
where
\begin{align*}
&I=\|\int_{\mathbb{R}^n}\big(\prod_{(i,j)\in S} (b_i(x+t)-b_i(y_j))-\prod_{(i,j)\in S} (b_i(x)-b_i(y_j))\big)K_{\delta} (x,y_1,
\dots,y_m)\\
&\quad \times f_1(y_1)\dots f_m(y_m)d\vec y\|_B,\\
&II=\|\int_{\mathbb{R}^n}\prod_{(i,j)\in S} (b_i(x+t)-b_i(y_j))\big(K_{\delta} (x+t,y_1,\dots,y_m)-K_{\delta} (x,y_1,
\dots,y_m)\big)\\
&\quad \times f_1(y_1)\dots f_m(y_m)d\vec y\|_B.
\end{align*}
Note that
\begin{align*}
&\prod_{(i,j)\in S} (a_{i,j}+b_{i,j})-\prod_{(i,j)\in S}b_{i,j}
=\sum_{D\subsetneq S}\prod_{(i,j)\in D}b_{i,j}\prod_{(i,j)\in S\setminus D}a_{i,j}.
\end{align*}
Let $a_{i,j}=b_i(x+t)-b_i(x),b_{i,j}=b_i(x)-b_i(y_j)$, one has
\begin{align*}
&\prod_{(i,j)\in S} (b_i(x+t)-b_i(y_j))-\prod_{(i,j)\in S} (b_i(x)-b_i(y_j))\\
&\quad=\sum_{D\subsetneq S}\prod_{(i,j)\in D} (b_i(x)-b_i(y_j))\prod_{(i,j)\in S\setminus D} (b_i(x+t)-b_i(x)).
\end{align*}
Therefore, we may continue to estimate $I.$
\begin{align*}
I\leq\sum_{D\subsetneq S}\prod_{(i,j)\in S\setminus D}|b_i(x+t)-b_i(x)|\|\int_{\mathbb{R}^n}&\prod_{(i,j)\in D} (b_i(x+t)-b_i(y_j))K_{\delta} (x,y_1,
\dots,y_m)\\
&\quad \quad \quad f_1(y_1)\dots f_m(y_m)d\vec y\|_B.
\end{align*}
Furthermore, since
\begin{align*}
\prod_{(i,j)\in D} (b_i(x+t)-b_i(y_j))=\sum_{E\subseteq D} (-1)^{|E|^c}\prod_{(i,j)\in E}b_i(x+h)\prod_{(i,j)\in D\setminus E}b_i(y_j),\end{align*}
we have
\begin{align*}
I\leq&\sum_{D\subsetneq S}\sum_{E\subseteq D}\prod_{(i,j)\in S\setminus D}|b_i(x+t)-b_i(x)|\prod_{(i,j)\in E}|b_i(x+h)|\\
&\quad\times\|\int_{\mathbb{R}^n}K_{\delta} (x,y_1,
\dots,y_m)\prod_{(i,j)\in D\setminus E}b_i(y_j)\prod_{j=1}^mf_j(y_j)d\vec y\|_B\\
&=\sum_{D\subsetneq S}\sum_{E\subseteq D}\prod_{(i,j)\in S\setminus D}|b_i(x+t)-b_i(x)|\prod_{(i,j)\in E}|b_i(x+h)|\\
&\quad \times\|T_{\delta} (f_1\prod_{i:(i,1)\in D\setminus E}b_i,\dots,f_m\prod_{i:(i, m)\in D\setminus E}b_i)(x)\|_B.
\end{align*}
By \eqref{inthm-bounds-Tdeltab}, we have 
\begin{align}
\|I\|_{L^p(\nu_{\vec\omega})}&\leq\sum_{D\subsetneq S}\sum_{E\subseteq D}\prod_{(i,j)\in S\setminus D}|t|\|\nabla b_i\|_{L^{\infty}}\prod_{(i,j)\in E}\|b_i\|_{L^{\infty}}\nonumber\\
&\quad \times\|T_{\delta} (f_1\prod_{i:(i,1)\in D\setminus E}b_i,\dots,f_m\prod_{i:(i, m)\in D\setminus E}b_i)(x)\|_{L^p_B(\nu_{\vec\omega})}\nonumber\\
&\leq C\sum_{D\subsetneq S}\sum_{E\subseteq D}\prod_{(i,j)\in S\setminus D}|t|\|\nabla b_i\|_{L^{\infty}}\prod_{(i,j)\in E}\|b_i\|_{L^{\infty}}\prod_{j=1}^m\|f_j\prod_{i:(i,j)\in D\setminus E}b_i\|_{L^{p_j} (\omega_j)}\nonumber\\
&\leq C\sum_{D\subsetneq S}\prod_{(i,j)\in S\setminus D}|t|\|\nabla b_i\|_{L^{\infty}}\prod_{(i,j)\in D}\|b_i\|_{L^{\infty}}\prod_{j=1}^m\|f_j\|_{L^{p_j} (\omega_j)}. \label{ineq2-ch7}
\end{align}
If $|t|\leq \frac12\delta$, the smoothness of $K_{\delta}$ yields that
\begin{align*}
&II\leq C\prod_{(i,j)\in S}\|b_i\|_{L^{\infty}}|t|\int_{\sum_{j=1}^m|x-y_j|\geq \delta}\frac{\prod_{j=1}^m|f_j(y_j)|}{\big(\sum_{j=1}^m|x-y_j|\big)^{nm+1}}d\vec y.
\end{align*}
By Lemma \ref{yl2-ch7}, we have
\begin{align*}
&II\leq C\prod_{(i,j)\in S}\|b_i\|_{L^{\infty}}|t|\frac 1\delta\mathcal{M} (\vec f)(x).
\end{align*}
Thus
\begin{align}\label{ineq3-ch7}
\|II\|_{L^p(\nu_{\vec\omega})}&\leq
C\prod_{(i,j)\in S}\|b_i\|_{L^{\infty}}|t|\frac 1\delta
\prod_{j=1}^m\|f_j\|_{L^{p_j} (\omega_j)}.
\end{align}
(iii) follows by combining (\ref{ineq2-ch7}) and (\ref{ineq3-ch7}). Hence, we completed the proof of Theorem \ref{dl5-ch7}, and finished the proof of Theorem \ref{thmmain}.
\end{proof}

\end{document}